\documentclass[11pt,reqno]{amsart}
\usepackage[T1]{fontenc}
\usepackage{lmodern}
\usepackage[a4paper,textwidth=160mm,textheight=232mm,centering]{geometry}
\usepackage{amssymb,mathtools,mathrsfs,microtype}
\usepackage[hidelinks]{hyperref}
\usepackage{enumitem}
\hypersetup{
 pdftitle={Equal sums and close divisors},
 pdfsubject={Entropy thresholds and the spacing of divisors},
 pdfkeywords={close divisors, logarithmic random sets, entropy, subset sums}
}
\allowdisplaybreaks[1]
\theoremstyle{plain}
\newtheorem{theorem}{Theorem}

\newtheorem{lemma}[theorem]{Lemma}
\newtheorem{conjecture}[theorem]{Conjecture}

\theoremstyle{definition}
\newtheorem{definition}[theorem]{Definition}

\theoremstyle{remark}

\newcommand{\N}{\mathbb N}
\newcommand{\Q}{\mathbb Q}
\newcommand{\R}{\mathbb R}
\newcommand{\Z}{\mathbb Z}
\newcommand{\Prob}{\mathbb P}
\newcommand{\E}{\mathbb E}
\newcommand{\zero}{\mathbf{0}}
\newcommand{\one}{\mathbf{1}}
\newcommand{\A}{\mathbf A}
\newcommand{\V}{\mathcal V}
\newcommand{\W}{\mathcal W}
\newcommand{\cc}{\mathbf c}
\newcommand{\mm}{\boldsymbol\mu}
\newcommand{\supp}{\operatorname{supp}}

\newcommand{\diam}{\operatorname{diam}}
\newcommand{\eps}{\varepsilon}
\title[Equal subset sums and close divisors]{Equal subset sums and close divisors}

\author{Jianfeng Hou}
\address{Center for Discrete Mathematics,
Fuzhou University, Fuzhou, Fujian, China}
\email{jfhou@fzu.edu.cn}
\author{Hongbin Zhao}
\address{Center for Discrete Mathematics, Fuzhou University, Fuzhou, Fujian, China}
\email{hbzhao2024@163.com}
\date{}
\begin{document}
\begin{abstract}
For $k\geq2$, let $\alpha_k$ be the supremum of the exponents $a$ for which almost every integer $n$ has $k$ distinct divisors in a multiplicative interval of relative length $(\log n)^{-a}$. Select each positive integer $i$ independently with probability $1/i$, forming a random set $\A$, and let $\beta_k$ be the supremum of the $c<1$ for which, with probability tending to one as $D\to\infty$, the set $\A\cap[D^c,D]$ has $k$ distinct subsets with the same sum. We prove that $\alpha_k=\beta_k/(1-\beta_k)$, resolving a conjecture of Ford, Green and Koukoulopoulos [Invent.\ Math.\ \textbf{232} (2023), 1027--1160]. We also prove that their weak and strict entropy thresholds coincide. The proof combines flag refinement and entropy concavity with an upper bound for approximate subset sums that is uniform in arbitrary translations. A model with independent geometric prime exponents then transfers this bound to divisors.
\end{abstract}
\maketitle

\section{Introduction}

For an integer $k\geq2$, let $\alpha_k$ be the supremum of the real numbers $a$ such that, for almost every integer $n\geq2$, there are divisors $d_1<\cdots<d_k$ of $n$ satisfying
\[
 d_k\leq d_1\bigl(1+(\log n)^{-a}\bigr).
\]
Here \emph{almost every} refers to natural density: the number of exceptions up to $x$ is $o(x)$.
The problem is to determine how tightly a prescribed number of divisors can be packed for almost all integers.

For $k=2$, Erd\H{o}s conjectured that $\alpha_2=\log3-1$.
This was proved by Erd\H{o}s and Hall \cite{EH} (upper bound) and Maier and Tenenbaum \cite{MT84} (lower bound). Maier and Tenenbaum \cite{MT85,MT09} subsequently studied larger clusters by combining close divisors from successive ranges of prime factors. Their work gives lower bounds for general $k$, as well as the upper bound $\alpha_k\leq(\log2)/(k+1)$ for $k\geq3$ \cite{MT09}. A systematic account of these problems can be found in Hall and Tenenbaum \cite{HT}.

Ford, Green and Koukoulopoulos \cite{FGK} developed a framework for constructing close divisors that allows for general patterns of shared prime factors. The underlying problem concerns equal subset sums in a random set $\A\subset\N$, where each integer $i$ is included independently with probability $1/i$.
These probabilities reflect the relation $\sum_{e^i<p\leq e^{i+1}}1/p\sim1/i$ as $i\to\infty$, with the sum taken over primes.
For $k\geq2$, let $\beta_k$ be the supremum of the numbers $c<1$ for which, with probability tending to one as $D\to\infty$, there are pairwise distinct subsets $B_1,\ldots,B_k\subset\A\cap[D^c,D]$ satisfying $\sum_{a\in B_1}a=\cdots=\sum_{a\in B_k}a$.
Ford, Green and Koukoulopoulos \cite[Theorem~6]{FGK} proved that
\begin{equation}\label{eq:known-lower}
 \alpha_k\geq\frac{\beta_k}{1-\beta_k},
\end{equation}
and posed the following equality conjecture.

\begin{conjecture}[Ford--Green--Koukoulopoulos {\cite{FGK}}]
For every integer $k \ge 2$, $\alpha_k = \frac{\beta_k}{1 - \beta_k}$.
\end{conjecture}

Our main result confirms this conjecture.
\begin{theorem}\label{thm:main}
For every integer $k\geq2$, $\alpha_k = \frac{\beta_k}{1 - \beta_k}$.
\end{theorem}

Our proof also resolves a question about the entropy formulation of the random-set problem.
Ford, Green and Koukoulopoulos \cite{FGK} associate systems of rational subspaces and probability measures on $\{0,1\}^k$ to configurations of $k$ subsets. These systems record the membership patterns of elements in the subsets across successive size ranges.
Their upper bound for equal subset sums leads to a weak entropy condition, while their construction of equal subset sums requires its strict version.
Writing $\gamma_k$ and $\widetilde\gamma_k$ for the respective optimal endpoints, they  prove that
\begin{equation}\label{eq:threshold-sandwich}
 \widetilde\gamma_k\leq\beta_k\leq\gamma_k.
\end{equation}

The precise definitions will be recalled in Section~\ref{sec:entropy}.
Theorem~\ref{thm:strict} establishes that $\widetilde{\gamma}_k=\gamma_k$, answering the question in \cite[Remark~3.1(a)]{FGK} and hence identifying both thresholds with $\beta_k$. More precisely, every weakly feasible system can be replaced by a strictly feasible system with arbitrarily small loss in its endpoint. The proof combines flag refinement with entropy concavity.

To establish the arithmetic upper bound, Lemma \ref{lem:translated} gives a power-saving estimate for approximate subset sums above the threshold $\beta_k$, uniformly in arbitrary translations. These translations and approximation errors arise from splitting divisors into small- and large-prime parts and rounding prime logarithms. A model with independent geometric prime exponents then transfers the estimate to divisors, yielding the reverse of \eqref{eq:known-lower}.

All logarithms are natural. For a nonempty finite subset $F$ of $\R$, write $\diam F=\max F-\min F$.
Unless a dependence is specified, constants may depend on $k$ and the fixed exponents, but not on the variables tending to infinity.
For an event $E$, let $\mathbb I_E$ denote its indicator.

\section{Entropy thresholds and translated sums}\label{sec:entropy}

Write $Q_k=\{0,1\}^k$, $\zero=(0,\ldots,0)$ and $\one=(1,\ldots,1)$.
We regard $Q_k$ as a subset of $\mathbb Q^k$, and call its elements
\emph{cube vectors}.
Following \cite{FGK}, we recall the relevant terminology on flags and systems.
An $r$-step \emph{flag} is a nested sequence $\mathcal V=(V_0,\ldots,V_r)$ of rational subspaces with $\mathbb Q\mathbf 1=V_0\leq V_1\leq\cdots\leq V_r\leq\mathbb Q^k$. The inclusions are not required to be strict. We call the flag \emph{complete} if $\dim V_j=\dim V_{j-1}+1$ for every $1\leq j\leq r$, equivalently if $\dim V_j=j+1$ for every $j$.
\begin{definition}
    A \emph{system} is a triple $(\V,\cc,\mm)$ such that:
    \begin{enumerate}[label=(\alph*)]
        \item $\V$ is an $r$-step flag whose members $V_j$ are distinct and spanned by elements of $Q_k$;
        \item $\V$ is \emph{nondegenerate}, meaning that $V_r$ is not contained in any of the subspaces $\{v\in\mathbb Q^k:v_s=v_t\}$ with $s\ne t$;
        \item $\cc=(c_1,\ldots,c_{r+1})$ satisfies $1\geq c_1\geq\cdots\geq c_{r+1}\geq0$;
        \item $\mm=(\mu_1,\ldots,\mu_r)$ is an $r$-tuple of probability measures on $Q_k$;
        \item $\supp\mu_j\subseteq V_j\cap Q_k$ for every $1\leq j\leq r$.
    \end{enumerate}
    We call $c_{r+1}$ the \emph{endpoint} of the system $(\mathcal{V}, \mathbf{c}, \boldsymbol{\mu})$, and say that the system is \emph{complete} if its underlying flag $\mathcal{V}$ is complete.
\end{definition}

A subflag $\W\leq\V$ is a sequence of rational subspaces $W_0=\Q\one\leq W_1\leq\cdots\leq W_r$ such that $W_j\subseteq V_j$ for every $j=0,\ldots,r$.
Its inclusions need not be strict, and its spaces need not be spanned by cube vectors. We call the subflag proper when $W_j\ne V_j$ for at least one index $j$.
For a probability measure $\mu$ on $Q_k$ and a rational subspace $W\subseteq\Q^k$, each coset $C$ of $W$ has mass $\mu(C)=\sum_{\omega\in C\cap Q_k}\mu(\omega)$.
The entropy of these coset masses is
\[
 H_\mu(W)=-\sum_{C\in\Q^k/W}\mu(C)\log\mu(C),
 \qquad 0\log0=0.
\]
Only cosets meeting the finite set $Q_k$ can have positive mass.
Define the energy
\begin{equation}\label{eq:energy}
 e_{\cc,\mm}(\W)
 =\sum_{j=1}^r(c_j-c_{j+1})H_{\mu_j}(W_j)
   +\sum_{j=1}^r c_j\dim(W_j/W_{j-1}).
\end{equation}
Since $H_{\mu_j}(V_j)=0$, the value at $\V$ is just the second sum in \eqref{eq:energy}.
The weak entropy condition is $e_{\cc,\mm}(\W)\geq e_{\cc,\mm}(\V)$ for every subflag.
The strict condition requires strict inequality for every $\W\ne\V$. The suprema of $c_{r+1}$ over the weak and strict systems are denoted by $\gamma_k$ and $\widetilde\gamma_k$, respectively. We call a system weakly or strictly feasible when it satisfies the corresponding entropy condition.

The next theorem gives the approximation needed to obtain equality in \eqref{eq:threshold-sandwich}.

\begin{theorem}\label{thm:strict}
Suppose a system with endpoint $b=c_{r+1}$ satisfies the weak entropy condition. For every $\eps>0$, its flag admits a complete refinement supporting a system that satisfies the strict entropy condition and has endpoint greater than $b-\eps$.
Consequently $\widetilde\gamma_k=\gamma_k$.
\end{theorem}

\begin{proof}
Fix a subflag $\W\leq\V$. Put $t_j=c_j-c_{j+1}$ and
$d_j=\dim V_j-\dim W_j$ for $j=1,\ldots,r$, and set $d_0=0$.
The difference of the dimension terms in the two energies is
$-\sum_{j=1}^r c_j(d_j-d_{j-1})$.
Since $c_{r+1}=b$, summation by parts gives
$\sum_{j=1}^r c_j(d_j-d_{j-1})=\sum_{j=1}^r t_jd_j+bd_r$.
Thus
\begin{equation}\label{eq:defect}
 \Phi_{\cc,\mm}(\W)
 :=e_{\cc,\mm}(\W)-e_{\cc,\mm}(\V)
 =\sum_{j=1}^r t_j\bigl(H_{\mu_j}(W_j)-d_j\bigr)-b\,d_r.
\end{equation}

We first refine the flag. Choose a basis of $V_i/V_{i-1}$ from images of cube vectors, and insert the subspaces obtained by adding these vectors one at a time. In the block refining the step $V_{i-1}<V_i$, give every new step threshold $c_i$.
Assign $\mu_i$ to the final step in the block and the point mass at $\zero$ to all its other steps. The latter steps have interval length zero. All the new subspaces are cube-spanned, and the terminal space is unchanged.

For a subflag of this refinement, retain its spaces at the ends
of the blocks. They form a subflag of the original flag.
The entropy terms inside a block vanish except at its last step,
and the dimension terms telescope because the thresholds throughout
the block are equal. Thus its energy is exactly the energy
of the compressed subflag. The same equality holds for the two
ambient flags. Weak feasibility and the endpoint are therefore
preserved. Relabel the refined system as $(\V,\cc,\mm)$, with $r$
steps, and use \eqref{eq:defect} for its subflags. We now have
$\dim V_j=j+1$ and $1\leq r\leq k-1$.

We construct a strict reference system on this complete flag.
For each $j$, choose $\omega_j\in(V_j\cap Q_k)\setminus V_{j-1}$,
and let $\nu_j$ give mass $1/3$ to each of
$\zero,\omega_j,\one-\omega_j$.
These vectors belong to $V_j\cap Q_k$. Modulo $V_{j-1}$, their
classes are $[\zero]$, $[\omega_j]$ and $-[\omega_j]$. They are distinct
because $[\omega_j]\ne[\zero]$ in the rational vector space $V_j/V_{j-1}$.
Hence $H_{\nu_j}(V_{j-1})=\log3$.
Set $h=\log3-1>0$. Choose $q\in(0,1)$ so small that
$rq/(1-q)<h/2$, and put
\[
 \tau_j=\frac{q^{j-1}}{\sum_{i=1}^r q^{i-1}},
 \qquad c_j^*=\sum_{i=j}^r\tau_i,
 \qquad c_{r+1}^*=0.
\]
Let $\W$ be a proper subflag, and let $j$ be the first index
at which $W_j\ne V_j$. Since the flag is complete and
$W_{j-1}=V_{j-1}$, necessarily $W_j=V_{j-1}$ and $d_j=1$.
The terms before $j$ in \eqref{eq:defect} vanish; the term at
$j$ is $h\tau_j$. For $i>j$ we use
$H_{\nu_i}(W_i)\geq0$ and $d_i\leq\dim V_i-1=i\leq r$,
so $H_{\nu_i}(W_i)-d_i\geq-r$. It follows that
\[
 \Phi_{\cc^*,\boldsymbol\nu}(\W)
 \geq h\tau_j-r\sum_{i>j}\tau_i
 \geq \tau_j\left(h-\frac{rq}{1-q}\right)>0.
\]
The reference system therefore satisfies the strict entropy condition.

For $0<\theta<1$, mix the interval-weighted measures by setting
\[
 t_j^{(\theta)}=(1-\theta)t_j+\theta\tau_j,\qquad
 \mu_j^{(\theta)}
 =\frac{(1-\theta)t_j\mu_j+\theta\tau_j\nu_j}
        {t_j^{(\theta)}},\qquad
 b^{(\theta)}=(1-\theta)b.
\]
Every $t_j^{(\theta)}$ is positive.
Set $c_{r+1}^{(\theta)}=b^{(\theta)}$ and, for $1\leq j\leq r$,
define $c_j^{(\theta)}=b^{(\theta)}+\sum_{i=j}^r t_i^{(\theta)}$.
Then $c_j^{(\theta)}=(1-\theta)c_j+\theta c_j^*$ for every $j$,
so the new
thresholds satisfy the required order and upper bound $1$.
Both $\mu_j$ and $\nu_j$ are supported on $V_j\cap Q_k$, so
$\supp\mu_j^{(\theta)}\subseteq V_j\cap Q_k$.
Entropy is concave on a finite probability simplex
\cite[Section~2.7]{CT}. Apply this to the induced coset distributions
and multiply by $t_j^{(\theta)}$.
For every subspace $W_j\subseteq V_j$, this gives
\[
 t_j^{(\theta)}H_{\mu_j^{(\theta)}}(W_j)
 \geq(1-\theta)t_jH_{\mu_j}(W_j)
       +\theta\tau_jH_{\nu_j}(W_j).
\]
The other terms in \eqref{eq:defect} are linear in the interval lengths and the endpoint. Consequently, for every proper subflag $\W\leq\V$,
\begin{align*}
 \Phi_{\cc^{(\theta)},\mm^{(\theta)}}(\W)
 &\geq(1-\theta)\Phi_{\cc,\mm}(\W)
       +\theta\Phi_{\cc^*,\boldsymbol\nu}(\W)>0.
\end{align*}
The first defect on the right is nonnegative by the weak condition, and the second is positive by the strict reference construction.
Choose $\theta$ small enough that $\theta b<\eps$.
The new endpoint is then $b^{(\theta)}=b-\theta b>b-\eps$, which proves the approximation assertion.
Taking suprema gives $\widetilde\gamma_k\geq\gamma_k$; the reverse inequality holds because every strict system is weak.
\end{proof}

Together with \eqref{eq:threshold-sandwich}, Theorem~\ref{thm:strict} proves $\widetilde\gamma_k=\beta_k=\gamma_k$.
These numbers are positive by \cite{FGK}, and are less than one.
To prove the latter assertion, take any weakly feasible system and write $b=c_{r+1}$ for its endpoint.
For the constant subflag defined by $W_j=\Q\one$ for every $j=0,\ldots,r$, all dimension increments vanish. Since a distribution on $Q_k$ has entropy at most $k\log2$, its energy is at most $k\log2\sum_{j=1}^r(c_j-c_{j+1}) =k\log2(c_1-b)\leq k\log2(1-b)$. On the other hand, $e_{\cc,\mm}(\V)\geq b(\dim V_r-1)\geq b$.
The last inequality uses nondegeneracy, which implies $\dim V_r\geq2$.
The weak condition therefore gives $b\leq k\log2(1-b)$, and hence $b\leq k\log2/(1+k\log2)<1$.

For a finite set $S\subset\N$, a vector $z=(z_1,\ldots,z_k)\in\R^k$ and a number $H\geq0$, write $\mathcal C_k(S,z,H)$ for the condition that there are pairwise distinct subsets $B_1,\ldots,B_k\subset S$ such that
\[
 \diam\left\{z_i+\sum_{a\in B_i}a:1\leq i\leq k\right\}\leq H.
\]
For fixed $c>0$, the sets $\A\cap[D^c,D]$ and $\A\cap(D^c,D]$ differ only if $D^c$ is an integer belonging to $\A$, an event of probability at most $D^{-c}$.

\begin{lemma}\label{lem:translated}
Fix an integer $k \ge 2$ and a real number $c \in (\beta_k, 1)$, and set $\mathbf{A}_D = \mathbf{A} \cap (D^c, D]$. 
There exist positive constants $\eta$, $C_{k,c}$, and $D_0$, depending only on $k$ and $c$, and for each $D \ge D_0$ an event $\mathcal{G}_D$ determined by $\mathbf{A}_D$ satisfying $\mathbb{P}(\mathcal{G}_D) \to 1$ as $D \to \infty$, such that for all $z \in \mathbb{R}^k$ and $H \ge 0$,
\begin{equation}\label{eq:translated}
\mathbb{P}\big(\mathcal{G}_D \cap \mathcal{C}_k(\mathbf{A}_D, z, H)\big) \le C_{k,c} (1 + H)^{k-1} D^{-\eta}.
\end{equation}
\end{lemma}

\begin{proof}
Put $L=\log D$. Let $\mathcal G_D$ be the event that the following
two bounds hold simultaneously for every pair of real numbers
$u,v$ satisfying $c\leq u\leq v\leq1$:
\[
 (v-u)L-L^{3/4}
 \leq |\A_D\cap(D^u,D^v]|
 \leq (v-u)L+L^{3/4}.
\]
For each integer $n$ with $D^c<n\leq D$, put
$X_n=\mathbb I_{\{n\in\A\}}-1/n$.
These variables are independent and have mean zero, and
$\sum_{D^c<n\leq D}\operatorname{Var}(X_n)
\leq\sum_{D^c<n\leq D}1/n\ll L$.
Kolmogorov's maximal inequality gives
\[
 \Prob\left(\max_{D^c\leq x\leq D}
 \left|\sum_{D^c<n\leq x}X_n\right|\le\tfrac13L^{3/4}\right)
 =1- O(L^{-1/2}).
\]
On this event, every centred interval sum $\sum_{D^u < n \le D^v} X_n$ has absolute value at most $2L^{3/4}/3$, being the difference of two partial sums. Together with the uniform estimate $\sum_{D^u < n \le D^v} 1/n = (v - u)L + O(D^{-c})$, this implies that the inequalities defining $\mathcal{G}_D$ hold for all sufficiently large $D$. Thus $\mathbb{P}(\mathcal{G}_D) = 1 - O(L^{-1/2}) \to 1$ as $D \to \infty$.

Taking $u=c$ and $v=1$ in the definition of $\mathcal G_D$ gives
$|\A_D|\leq(1-c)L+L^{3/4}\leq2L$ for sufficiently large $D$.
Fix a realization for which both $\mathcal G_D$ and
$\mathcal C_k(\A_D,z,H)$ hold, and choose subsets
$B_1,\ldots,B_k$ witnessing the latter condition.
For each $a\in\A_D$, its membership vector $\omega(a)\in Q_k$
has $i$-th coordinate $1$ when $a\in B_i$ and $0$ otherwise.
Start with $V_0=\Q\one$.
Once $V_{j-1}$ has been constructed, stop if every membership
vector belongs to $V_{j-1}$. Otherwise define
\[
 x_j=\max\{a\in\A_D:\omega(a)\notin V_{j-1}\},\qquad
 \omega^j=\omega(x_j),\qquad
 V_j=V_{j-1}+\Q\omega^j.
\]
We call $x_j$ the $j$-th pivot.

Each chosen membership vector increases the dimension by one.
The procedure therefore stops after $r\leq k-1$ choices, with
$\dim V_j=j+1$. It makes at least one choice: if every membership
vector were in $V_0\cap Q_k=\{\zero,\one\}$, all the $B_i$ would
coincide. After choosing $x_j$, all membership vectors at elements
$a\geq x_j$ belong to $V_j$, so the next pivot satisfies
$x_{j+1}<x_j$.
At termination every membership vector belongs to $V_r$.
For each $s\ne t$, the distinctness of $B_s$ and $B_t$ supplies a
membership vector with unequal $s$-th and $t$-th coordinates.
Consequently $V_r$ is contained in none of the coordinate hyperplanes
$\{v\in\Q^k:v_s=v_t\}$ with $s\ne t$, and $\V=(V_0,\ldots,V_r)$ is a complete
nondegenerate flag.

For each $j=1,\ldots,r$, let $c_j$ be the smallest point of the
grid $1+L^{-1}\Z$ that is at least $\log x_j/L$, and set $c_{r+1}=c$.
Since the grid contains $1$, the rounded thresholds satisfy
$1\geq c_1\geq\cdots\geq c_r>c$.
Moreover, $0\leq c_j-\log x_j/L<1/L$, which gives
$D^{c_j}/e<x_j\leq D^{c_j}$.
Remove the pivots and put $S=\A_D\setminus\{x_1,\ldots,x_r\}$.
Since $S\subseteq\A_D$, we have $|S|\leq2L$ and, for
$c\leq u\leq v\leq1$,
\begin{equation}\label{eq:remaining-counts}
 |S\cap(D^u,D^v]|\leq(v-u)L+L^{3/4}.
\end{equation}
For each $j=1,\ldots,r$, define the layer
$S_j=S\cap(D^{c_{j+1}},D^{c_j}]$.
If $S_j$ is nonempty, then for each $\omega\in Q_k$ define
\[
 \mu_j(\omega)
 =\frac{|\{a\in S_j:\omega(a)=\omega\}|}{|S_j|}.
\]
If $S_j$ is empty, let $\mu_j$ be the point mass at $\zero$.
For $j<r$, every $a\in S_j$ satisfies
$a>D^{c_{j+1}}\geq x_{j+1}$, so the maximality of $x_{j+1}$
gives $\omega(a)\in V_j$.
Every membership vector belongs to $V_r$ when the procedure stops,
which gives the same conclusion for $j=r$.
Thus $\supp\mu_j\subseteq V_j\cap Q_k$ for every $j$.
Similarly, if $a\in S$ and $a>D^{c_1}$, then $a>x_1$, and the
maximality of $x_1$ gives $\omega(a)\in V_0$.
The resulting $(\V,\cc,\mm)$ is therefore a complete system
with endpoint $c$.
For every $j=1,\ldots,r$, the upper bound in
\eqref{eq:remaining-counts}, with $u=c_{j+1}$ and $v=c_j$, gives
$|S_j|\leq(c_j-c_{j+1})L+L^{3/4}$.

The data $(\V,\cc,\mm,\omega^1,\ldots,\omega^r)$ have only
$L^{O_k(1)}$ possible values.
To see this, the flag and the pivot membership vectors are chosen
from the fixed finite cube, and each $c_j$ has $O(L)$ possible
values. For each layer, $\mu_j$ is determined by the $2^k$
nonnegative integer counts of its membership vectors.
Every count is at most $2L$, so there are at most
$(\lfloor2L\rfloor+1)^{2^k}$ choices of these counts per layer.
There are at most $k-1$ layers.

Fix one such choice of data, and then fix a remaining set $S$
that occurs with these data. We will sum over these sets after
counting the compatible pivot tuples.
Let $\pi:\R^k\to\R^{k-1}$ be the map
$\pi(y_1,\ldots,y_k)=(y_2-y_1,\ldots,y_k-y_1)$, whose kernel
is $\R\one$.
Let $\mathscr L(S)$ consist of the vectors
$\pi(\sum_{a\in S}a\psi(a))$, where $\psi:S\to Q_k$ ranges over
assignments with empirical measure $\mu_j$ on each layer $S_j$ and
with $\psi(a)\in\{\zero,\one\}$ for $a>D^{c_1}$.
The elements above $D^{c_1}$ contribute zero after applying $\pi$;
on each layer $S_j$, the support condition gives $\psi(a)\in V_j$.
We shall show that there is a constant $C_k>0$, depending only on
$k$, such that
\begin{equation}\label{eq:cardinality}
 |\mathscr L(S)|
 \leq \exp(C_kL^{3/4})
 D^{\min_{\W\leq\V}e_{\cc,\mm}(\W)}.
\end{equation}

For a fixed ambient flag $\V$, the energy of a subflag $\W$
depends only on the dimensions $\dim W_j$ and the partitions of
$Q_k$ defined by
$\omega\sim_j\omega'$ if $\omega-\omega'\in W_j$.
Indeed, the partitions determine every coset mass used in the
entropy terms, and the dimensions determine the second sum in
\eqref{eq:energy}.
There are finitely many such choices of dimensions and partitions.
Choose one rational subflag representing each choice that occurs.
Every subflag has the same energy as one of these representatives,
for every admissible $(\cc,\mm)$.
Hence the minimum in \eqref{eq:cardinality} is attained among
this finite collection.

Fix one representative subflag $\W$.
Choose a rational basis of $W_r$ whose first vector is $\one$ and
whose first $\dim W_j$ vectors form a basis of $W_j$ for each
$j=0,\ldots,r$.
For each $j$, extend those first $\dim W_j$ vectors to a basis
of $V_j$. Define $P_j:V_j\to W_j$ by keeping their coordinates
and setting the added coordinates to zero.
Thus $P_j$ is the identity on $W_j$.
There are finitely many cube-spanned ambient flags and finitely
many representative subflags for each. Fix these bases and maps
for all of them before varying the thresholds, measures or set $S$.
Express every $P_j\omega$, $\omega\in V_j\cap Q_k$, in the
chosen basis of $W_r$. These finitely many rational coordinates
are bounded in absolute value and have a common positive integer
denominator, both depending only on $k$.

For $\omega\in V_j\cap Q_k$, write
$\omega=(I-P_j)\omega+P_j\omega$, where $I$ is the identity on $V_j$.
If $\omega-\omega'\in W_j$, then
$(I-P_j)(\omega-\omega')=0$. Hence $(I-P_j)\omega$ depends
only on the coset $\omega+W_j$.
For each coset $C$ of $W_j$ meeting $Q_k$, the fixed empirical
measure prescribes exactly $n_C=|S_j|\mu_j(C)$ elements in that
coset. These nonnegative integers sum to $|S_j|$.
The number of ways to assign these cosets to the elements of
$S_j$ is
\[
 \frac{|S_j|!}{\prod_C n_C!}
 \leq\prod_{C:n_C>0}\mu_j(C)^{-n_C}
 =\exp\bigl(|S_j|H_{\mu_j}(W_j)\bigr).
\]
The inequality follows from the term with these multiplicities in
the multinomial expansion of
$(\sum_C\mu_j(C))^{|S_j|}=1$.
Once all these cosets are assigned, the vector
$\sum_j\sum_{a\in S_j}a(I-P_j)\psi(a)$ is fixed.
Multiplying the bounds for the layers therefore bounds the number
of such vectors.
Using $|S_j|\leq(c_j-c_{j+1})L+L^{3/4}$ and
$H_{\mu_j}(W_j)\leq k\log2$ for each $j$, this product is at most
\[
 \begin{aligned}
 \exp\left(\sum_{j=1}^r|S_j|H_{\mu_j}(W_j)\right)
 &\leq \exp\bigl(rk(\log2)L^{3/4}\bigr)\,
 D^{\sum_{j=1}^r(c_j-c_{j+1})H_{\mu_j}(W_j)}.
 \end{aligned}
\]

Consider next the coordinates of
$\sum_{i=1}^r\sum_{a\in S_i}aP_i\psi(a)$ in the chosen basis
of $W_r$.
Take a basis vector added when the basis of $W_{j-1}$ was extended
to a basis of $W_j$.
Its coordinate receives no contribution from $S_i$ when $i<j$,
because $P_i\psi(a)\in W_i\subseteq W_{j-1}$.
When $i\geq j$, an element $a\in S_i$ satisfies
$a\leq D^{c_i}\leq D^{c_j}$.
Since the coordinates of $P_i\psi(a)$ are $O_k(1)$ and
$|S|\leq2L$, the coordinate of this sum is $O_k(LD^{c_j})$.
It is an integer multiple of the reciprocal of the fixed common
denominator, because every $a$ is an integer. It therefore has
$O_k(LD^{c_j})$ possible values.

There are $\dim W_j-\dim W_{j-1}$ such coordinates at step $j$.
The coordinate along $\one$ is removed by $\pi$, and all remaining
coordinates together determine the image under $\pi$.
Multiplying their numbers of possible values bounds the number
of possible images under $\pi$ of these $W_r$-valued contributions by
$O_k(L^{\dim W_r-1})D^{\sum_{j=1}^r c_j\dim(W_j/W_{j-1})}$.
Multiply this by the bound for the complementary contributions.
The exponent of $D$ is exactly $e_{\cc,\mm}(\W)$.
Since $\dim W_r-1\leq k-1$ and $r\leq k-1$, all remaining
factors are at most $\exp(C_kL^{3/4})$ for a sufficiently large
constant $C_k$ depending only on $k$.
Choose a representative subflag of minimum energy to obtain
\eqref{eq:cardinality}.

For each $v\in\mathscr L(S)$, a compatible pivot tuple must satisfy
\[
 \left\|v+\sum_{j=1}^r x_j\pi(\omega^j)+\pi(z)\right\|_\infty
 \leq H.
\]
Since $\one,\omega^1,\ldots,\omega^r$ are rational vectors independent over $\Q$ and $\ker\pi=\R\one$, the columns $\pi(\omega^1),\ldots,\pi(\omega^r)$ have rank $r$ over $\R$.
Choose $r$ rows giving a nonsingular square matrix $M$.
Its entries lie in $\{-1,0,1\}$, its determinant is a nonzero integer, and its cofactors are bounded in terms of $k$.
Thus $M^{-1}$ is bounded in terms of $k$.
Let $\mathbf b\in\R^r$ consist of the corresponding coordinates of $v+\pi(z)$, and write $\mathbf x=(x_1,\ldots,x_r)$.
The selected inequalities say $M\mathbf x+\mathbf b=\mathbf t$ for a vector $\mathbf t\in\R^r$ satisfying $\|\mathbf t\|_\infty\leq H$.
Thus $\mathbf x=-M^{-1}\mathbf b+M^{-1}\mathbf t$.
Each coordinate of $\mathbf x$ lies in an interval of length $O_k(H)$,
so there are $O_k((1+H)^r)$ possible integer pivot tuples.
The interval lengths, and hence this count, are independent of
the centre $-M^{-1}\mathbf b$ and of $v,z$.

Under the Bernoulli law, adjoining distinct pivots
$X=\{x_1,\ldots,x_r\}$ to a remaining set $S$ disjoint from $X$
changes its mass by the exact factor
\begin{equation}\label{eq:mass-ratio}
 \frac{\Prob(\A_D=S\cup X)}{\Prob(\A_D=S)}
 =\prod_{j=1}^r\frac{1/x_j}{1-1/x_j}
 =\prod_{j=1}^r\frac1{x_j-1}
 \ll_k D^{-\sum_j c_j}.
\end{equation}
The last inequality follows from $x_j>D^{c_j}/e$ and $x_j\geq D^c\to\infty$.
For fixed data and $S$, the number of compatible pivot tuples is at most a constant depending only on $k$ times $|\mathscr L(S)|(1+H)^r$.
For each tuple, \eqref{eq:mass-ratio} bounds the probability of the corresponding set $S\cup X$.
Combining it with \eqref{eq:cardinality}, and increasing $C_k$ to absorb the fixed multiplicative constants, bounds the sum of these probabilities by
\[
 \Prob(\A_D=S)\,(1+H)^{k-1}\exp(C_kL^{3/4})
 D^{\min_{\W\leq\V}e_{\cc,\mm}(\W)-\sum_{j=1}^r c_j}.
\]
Here $r\leq k-1$, and completeness gives $e_{\cc,\mm}(\V)=\sum_{j=1}^r c_j$.
For each fixed choice of data, sum over all remaining sets $S$ that occur with these data. Their unconditional Bernoulli probabilities $\Prob(\A_D=S)$ sum to at most $1$.
Then sum over the $L^{O_k(1)}$ possible choices of data.
Every realization of the event has at least one of the witnesses used in this construction, so these sums give an upper bound by the union bound.
The polynomial factor in $L$ can be absorbed by increasing $C_k$, because $\log L=o(L^{3/4})$.
We obtain
\begin{equation}\label{eq:probability-supremum}
 \begin{split}
 &\Prob\bigl(\mathcal G_D\cap\{\mathcal C_k(\A_D,z,H)\}\bigr)\\
 &\quad\leq(1+H)^{k-1}\exp(C_kL^{3/4})
 \sup_{\substack{(\V,\cc,\mm)\ {\rm complete}\\c_{r+1}=c}}
 D^{\min_{\W\leq\V}
       (e_{\cc,\mm}(\W)-e_{\cc,\mm}(\V))}.
 \end{split}
\end{equation}
The supremum ranges over the complete nondegenerate systems with endpoint $c$, and the same constant $C_k$ applies to all these systems and to every $z,H$.

For each complete nondegenerate flag $\V$, the admissible thresholds with endpoint $c$ form the closed bounded set $1\geq c_1\geq\cdots\geq c_r\geq c$.
Each $\mu_j$ ranges over the probability simplex on the finite set $V_j\cap Q_k$.
Their product is a compact parameter space.
The minimum $\min_{\W\leq\V}(e_{\cc,\mm}(\W)-e_{\cc,\mm}(\V))$ is a continuous function of $(\cc,\mm)$ on this space.
The finite representative subflags chosen above suffice for every $(\cc,\mm)$, so this is a finite minimum of continuous entropy expressions.
This minimum is strictly negative at every point: otherwise that point would be a weakly feasible system with endpoint $c$, implying $c\leq\gamma_k=\beta_k$.
Its maximum on the compact parameter space is therefore strictly negative.
There are only finitely many possible ambient flags, so their maxima have a common upper bound $-\eta_0$, where $\eta_0=\eta_0(k,c)>0$.
Substituting this bound in \eqref{eq:probability-supremum} gives $(1+H)^{k-1}\exp(C_kL^{3/4})D^{-\eta_0}$.
Since $L^{3/4}=o(L)$, this is at most $(1+H)^{k-1}D^{-\eta_0/2}$ for sufficiently large $D$.
Taking $\eta=\eta_0/2$ proves \eqref{eq:translated}.
\end{proof}

\section{From prime factors to divisors}

\begin{proof}[Proof of Theorem~\ref{thm:main}]
Fix $k\geq2$. It suffices to prove the reverse of
\eqref{eq:known-lower}.
We first transfer the density-one property to an independent model
and then show that close divisors in this model are rare above the
claimed threshold. We use Mertens' estimates and the prime number
theorem with its
classical zero-free-region error term, as in \cite{KoukoBook}.

Let $y\to\infty$ and put $T=\log y$.
For every prime $p\leq y$, let $N_p$ be a nonnegative integer-valued
random variable. Take these variables independently, with
$\Prob(N_p=j)=(1-1/p)p^{-j}$ for each integer $j\geq0$.
Set $M_y=\prod_{p\leq y}p^{N_p}$ and
$C_y=\prod_{p\leq y}(1-1/p)$.
An integer is called $y$-smooth if all its prime factors are at
most $y$; the integer $1$ is included.
For every $y$-smooth integer $m=\prod_{p\leq y}p^{v_p}$,
independence gives the exact formula
\begin{equation}\label{eq:prime-model}
 \Prob(M_y=m)
 =\prod_{p\leq y}(1-1/p)p^{-v_p}
 =\frac{C_y}{m}.
\end{equation}
For all other positive integers $m$, $\Prob(M_y=m)=0$.

For $j\geq1$, the geometric law gives $\Prob(N_p\geq j)=p^{-j}$
and $\E N_p=1/(p-1)$. Mertens' estimates imply
$C_y\asymp T^{-1}$ and
\[
 \E\log M_y=\sum_{p\leq y}\frac{\log p}{p-1}\ll T.
\]

For $a>0$, let $\mathcal E_y(a)$ be the event that $M_y$ has
$k$ distinct divisors in a multiplicative interval of relative
length $T^{-a}$.
We first show that $\Prob(\mathcal E_y(a))\to1$ whenever
$0<a<b$ and the density-one assertion defining $\alpha_k$
holds at exponent $b$.
Let $E_b$ be the exceptional set, including $1$, and let
$N_b(X)=|E_b\cap[1,X]|$ count its elements up to $X$.
The density assumption says $N_b(X)=o(X)$.
Partial summation gives
\[
 \sum_{\substack{m\leq X\\m\in E_b}}\frac1m
 =\frac{N_b(X)}X+\int_1^X\frac{N_b(t)}{t^2}\,dt
 =o(\log X).
\]
Fix real numbers $\rho,R$ satisfying $0<\rho<1<R$.
The mass formula \eqref{eq:prime-model} and the harmonic-sum
estimate give
$\Prob(M_y<y^\rho)\leq C_y\sum_{m<y^\rho}1/m
\ll\rho+T^{-1}$.
The same mass formula also gives
\[
 \Prob(M_y\in E_b,\ M_y\leq y^R)
 \leq C_y\sum_{\substack{m\leq y^R\\m\in E_b}}\frac1m=o(1).
\]
For the upper tail, Markov's inequality and $\E\log M_y\ll T$
give $\Prob(M_y>y^R)\leq\E\log M_y/(RT)\ll R^{-1}$.

For fixed $\rho,R$ and sufficiently large $y$, every integer
$m\geq y^\rho$ outside $E_b$ has the required divisors with
relative spread at most
$(\log m)^{-b}\leq\rho^{-b}T^{-b}\leq T^{-a}$.
Thus $\mathcal E_y(a)$ can fail only if $M_y<y^\rho$, or
$M_y>y^R$, or $M_y\in E_b$ with $M_y\leq y^R$.
The bounds above imply
$\limsup_{y\to\infty}\Prob(\mathcal E_y(a)^c)\ll\rho+R^{-1}$.
The implicit constant is independent of $\rho,R$.
Letting $\rho\downarrow0$ and $R\to\infty$ proves that
$\Prob(\mathcal E_y(a))\to1$.

We now prove $\Prob(\mathcal E_y(a))\to0$ under the condition
\begin{equation}\label{eq:exponent-choice}
 a>0,\qquad \frac{a}{1+a}>\beta_k.
\end{equation}
Fix such an $a$, choose $\beta_k<c<a/(1+a)$, and put $K=T^a$
and $D=KT$.
Set $m_0=\lceil K(\log T)^3\rceil$ and $y_0=\exp(m_0/K)$.
The cutoff $y_0$ is an endpoint of the logarithmic intervals below;
the factor $(\log T)^3$ will make the accumulated error from the
prime number theorem tend to zero.
Since $\log y_0=(\log T)^3+O(K^{-1})=o(T)$, we have $y_0<y$
for sufficiently large $y$.
The small-prime part of $M_y$ is the integer
$S=\prod_{p\leq y_0}p^{N_p}$.
The expectation formula for $N_p$ gives
$\E\log S\ll\log y_0\asymp(\log T)^3$.
Let $\tau(S)$ denote the number of positive divisors of $S$.
Using $N_p+1\leq2^{N_p}$ in
$\tau(S)=\prod_{p\leq y_0}(N_p+1)$ gives
\[
 \E\log\tau(S)
 \leq(\log2)\sum_{p\leq y_0}\frac1{p-1}
 \ll\log\log y_0\ll\log\log T.
\]
Markov's inequality therefore implies
$\Prob(\tau(S)>\exp(\sqrt{\log T}))
\ll(\log\log T)/\sqrt{\log T}=o(1)$.

We next estimate the probability that $S$ has two close divisors.
Put $\eps=T^{-a}$ and $X_0=\exp(T^{a/2})$.
First, Markov's inequality gives
$\Prob(S>X_0)\leq\E\log S/\log X_0$.
Suppose now that $S\leq X_0$ and two distinct divisors of $S$
have ratio in $(1,1+\eps]$.
Dividing them by their greatest common divisor gives positive
coprime integers $u<v\leq(1+\eps)u$.
Both divide $S$, so their product divides $S$ and
$u^2<uv\leq S\leq X_0$.
For any positive integer $q$ whose prime factors are all at most
$y_0$, independence and the identities $\Prob(N_p\geq j)=p^{-j}$
give $\Prob(q\mid S)=1/q$.
If $q$ has a larger prime factor, this probability is zero.
In particular, for each of the coprime pairs above,
$\Prob(uv\mid S)\leq1/(uv)$.

Apply the union bound to these coprime pairs and then drop
coprimality in the sum. With $u,v$ ranging over positive integers,
we obtain
\[
 \begin{split}
 &\Prob\bigl(S\text{ has two divisors with ratio in }(1,1+\eps]\bigr)\\
 &\quad\leq\frac{\E\log S}{\log X_0}
       +\sum_{u\leq\sqrt{X_0}}
        \ \sum_{u<v\leq(1+\eps)u}\frac1{uv}\\
 &\quad\leq\frac{\E\log S}{\log X_0}
       +\eps\sum_{u\leq\sqrt{X_0}}\frac1u\\
 &\quad\ll(\log T)^3T^{-a/2}+T^{-a}(1+T^{a/2}).
 \end{split}
\]
For each integer $u$, the interval $(u,(1+\eps)u]$ contains
$\lfloor\eps u\rfloor\leq\eps u$ integers, and $1/(uv)\leq u^{-2}$
for each of them. This proves the second inequality.
Both terms in the final bound tend to zero.
Let $\mathcal S_y$ be the event that $S$ has no two divisors with
ratio in $(1,1+T^{-a}]$ and that
$\tau(S)\leq\exp(\sqrt{\log T})$.
The two probability estimates give $\Prob(\mathcal S_y^c)=o(1)$.

Put $m_1=\lfloor KT\rfloor$.
For each integer $i$ satisfying $m_0\leq i<m_1$, let
$\mathcal P_i$ be the set of primes in the interval
$(e^{i/K},e^{(i+1)/K}]$, and write
$R_i=\sum_{p\in\mathcal P_i}1/p$.
In what follows, a prime $p\leq y$ is called selected when $N_p\geq1$.
Partial summation of the classical prime number theorem gives
$\sum_{p\leq x}1/p=\log\log x+M+
O(\exp(-c_*\sqrt{\log x}))$ for absolute constants $M$ and
$c_*>0$; see \cite{KoukoBook}.
Apply this estimate at the two endpoints of each interval.
For every integer $i$ with $m_0\leq i<m_1$, define
$\delta_i=R_i-\log(1+1/i)$.
Then, for an absolute constant $c_0>0$,
\begin{equation}\label{eq:bin-primes}
 R_i=\log(1+1/i)+\delta_i,\qquad
 |\delta_i|\ll\exp(-c_0\sqrt{i/K}).
\end{equation}
The main term is the difference
$\log((i+1)/K)-\log(i/K)$.
Each index satisfies $i/K\geq(\log T)^3$, and there are at most
$D=T^{a+1}$ indices. Hence
\[
 \sum_{i=m_0}^{m_1-1}|\delta_i|
 \ll\exp\bigl((a+1)\log T-c_0(\log T)^{3/2}\bigr)=o(1).
\]
All sums over $i$ without stated limits in the rest of the proof
range over the integers $m_0\leq i<m_1$.

For each of these indices, set $I_i=1$ if $\mathcal P_i$ contains
a selected prime, and $I_i=0$ otherwise. Write $q_i=\Prob(I_i=1)$.
The sets $\mathcal P_i$ are pairwise disjoint, so the indicators
$I_i$ are independent.
The union bound and the first two terms of inclusion--exclusion,
using independence for distinct primes, give
\[
 0\leq R_i-q_i
 \leq\sum_{\substack{p<q\\p,q\in\mathcal P_i}}\frac1{pq}
 \leq\tfrac12R_i^2.
\]
By \eqref{eq:bin-primes} and $\log(1+1/i)\leq1/i$,
\[
 \sum_iR_i^2
 \leq2\sum_{i=m_0}^{\infty}i^{-2}
       +2\Bigl(\sum_i|\delta_i|\Bigr)^2
 \ll m_0^{-1}+\Bigl(\sum_i|\delta_i|\Bigr)^2=o(1).
\]
Since $0\leq1/i-\log(1+1/i)\leq1/(2i^2)$, we obtain
\[
 \sum_{i=m_0}^{m_1-1}\left|q_i-\frac1i\right|
 \leq\tfrac12\sum_iR_i^2+\sum_i|\delta_i|
       +\tfrac12\sum_{i=m_0}^{\infty}i^{-2}=o(1).
\]
The same pair count bounds the probability that some interval
contains two distinct selected primes:
\[
 \Prob\bigl(\text{some }\mathcal P_i\text{ contains two selected primes}\bigr)
 \leq\sum_i\sum_{\substack{p<q\\p,q\in\mathcal P_i}}\frac1{pq}
 \leq\tfrac12\sum_iR_i^2=o(1).
\]
The probability that $N_p\geq2$ for at least one prime
$y_0<p\leq y$ is at most
$\sum_{y_0<p\leq y}\Prob(N_p\geq2)
=\sum_{y_0<p\leq y}p^{-2}=o(1)$.

The intervals defining $\mathcal P_i$ cover $(y_0,e^{m_1/K}]$.
The remaining interval up to $y$ is $(e^{m_1/K},y]$.
The probability that it contains a selected prime is at most
the sum of $1/p$ over its primes.
Applying the same reciprocal-prime estimate at its endpoints gives
\[
 \sum_{e^{m_1/K}<p\leq y}\frac1p
 =\log(KT/m_1)+O\bigl(\exp(-c_1\sqrt T)\bigr)
 =O\bigl(D^{-1}+\exp(-c_1\sqrt T)\bigr)
\]
for an absolute constant $c_1>0$.
Here $m_1=\lfloor D\rfloor$, so $\log(KT/m_1)=O(D^{-1})$.
This probability also tends to zero.

For sufficiently large $y$, every integer $i$ with $m_0\leq i<m_1$
lies in $(D^c,D]$.
The upper bound follows from $i<m_1\leq D$.
For the lower bound, the definitions give
\[
 \frac{\log m_0}{\log D}
 =\frac{a\log T+3\log\log T+o(1)}{(a+1)\log T}
 \longrightarrow\frac{a}{1+a}>c.
\]
For each such $i$, construct a Bernoulli variable $J_i$ of parameter
$1/i$ from $I_i$ as follows. Use auxiliary random variables that are
independent of all the $N_p$ and of one another.
If $q_i\geq1/i$, retain an outcome $I_i=1$ with conditional
probability $1/(iq_i)$, and put $J_i=0$ otherwise.
If $q_i<1/i$, retain every outcome $I_i=1$ and change an outcome
$I_i=0$ to $J_i=1$ with conditional probability
$(1/i-q_i)/(1-q_i)$.
In both cases $\Prob(J_i=1)=1/i$.
The probability of changing the original value is $q_i-1/i$ in
the first case and $1/i-q_i$ in the second, so
$\Prob(I_i\ne J_i)=|q_i-1/i|$.
The pairs $(I_i,J_i)$ are independent over the indices, because
each uses only the prime variables in its own set $\mathcal P_i$
and its own auxiliary randomness.

For every integer $i$ in $(D^c,D]$ outside the range
$m_0\leq i<m_1$, introduce an independent Bernoulli variable $J_i$
of parameter $1/i$, independent also of all previously defined
variables.
Then $\A_D=\{i\in\N:D^c<i\leq D,\ J_i=1\}$ has the law used
in Lemma~\ref{lem:translated}.
This set is independent of $S$, since its construction uses only
the $N_p$ with $y_0<p\leq y$ and auxiliary randomness, while
$S$ depends only on the $N_p$ with $p\leq y_0$.

Let $\mathcal B_y$ be the event that all the following conditions
hold. For every index $m_0\leq i<m_1$, we have $I_i=J_i$ and
$\mathcal P_i$ contains at most one selected prime. For every prime
$y_0<p\leq y$, we have $N_p\leq1$. Finally, no prime in
$(e^{m_1/K},y]$ is selected.
The preceding estimates and the union bound give
\[
 \begin{split}
 \Prob(\mathcal B_y^c)
 &\leq\sum_i|q_i-1/i|+\tfrac12\sum_i R_i^2\\
 &\quad+\sum_{y_0<p\leq y}p^{-2}
       +\sum_{e^{m_1/K}<p\leq y}p^{-1}
 =o(1).
 \end{split}
\]
For this coupled set $\A_D$, let $\mathcal G_D$ be the event from
the proof of Lemma~\ref{lem:translated}. Then
$\Prob(\mathcal G_D^c)=o(1)$, and $|\A_D|\leq2\log D$
on $\mathcal G_D$ for sufficiently large $D$.

Suppose $\mathcal E_y(a)\cap\mathcal S_y\cap\mathcal B_y$ holds,
and choose divisors $d_1<\cdots<d_k$ with $d_k/d_1\leq1+T^{-a}$.
Write $d_j=s_ju_j$, where $s_j\mid S$ and all prime factors of
$u_j$ exceed $y_0$.
If $u_j=u_\ell$ for some $j<\ell$, then
$1<s_\ell/s_j=d_\ell/d_j\leq1+T^{-a}$, contradicting
$\mathcal S_y$. Thus the integers $u_1,\ldots,u_k$ are pairwise
distinct.
On $\mathcal B_y$, each $u_j$ is a product of distinct selected
primes, with at most one from each set $\mathcal P_i$.
Let $B_j$ be the set of indices $i$ for which $u_j$ contains
the selected prime from $\mathcal P_i$.
There is at most one such prime per index, so the distinct $u_j$
give distinct sets $B_j$.
For each $i\in B_j$, selection implies $I_i=1$, and
$\mathcal B_y$ gives $J_i=I_i=1$. Hence $B_j\subseteq\A_D$.

For an index $i$ with a selected prime $p_i\in\mathcal P_i$,
write $K\log p_i=i+\theta_i$.
The endpoints of its defining interval give $0<\theta_i\leq1$.
Consequently, for every $j=1,\ldots,k$,
\[
 K\log d_j=K\log s_j+\sum_{i\in B_j}i
                         +\sum_{i\in B_j}\theta_i.
\]
The numbers on the left have diameter at most
$K\log(1+T^{-a})\leq KT^{-a}=1$.
Each final sum on the right lies between $0$ and $|B_j|$, hence
between $0$ and $|\A_D|$.
Thus the difference between any two of these error sums has
absolute value at most $|\A_D|$.
Subtracting them gives
\[
 \diam\left\{K\log s_j+\sum_{i\in B_j}i:1\leq j\leq k\right\}
 \leq1+|\A_D|.
\]
On $\mathcal G_D$ this diameter is at most the deterministic
number $H_0=1+2\log D$.

Fix a value $s$ of $S$ satisfying the two conditions defining
$\mathcal S_y$, and condition on $S=s$.
There are $\tau(s)^k$ ordered tuples $(s_1,\ldots,s_k)$ of
positive divisors of $s$, with repetitions allowed.
Each tuple determines the fixed vector
$z=(K\log s_1,\ldots,K\log s_k)$.
By the preceding construction, on $\mathcal B_y\cap\mathcal G_D$
the event $\mathcal E_y(a)$ implies
$\mathcal C_k(\A_D,z,H_0)$ for at least one of these tuples.
Since $S$ and $\A_D$ are independent, the conditional law of
$\A_D$ is unchanged, and Lemma~\ref{lem:translated} applies
to each fixed $z$. The union bound therefore gives
\[
 \begin{split}
 &\Prob\bigl(\mathcal E_y(a)\cap\mathcal B_y\cap\mathcal G_D
       \mid S=s\bigr)\\
 &\quad\leq\sum_{s_1,\ldots,s_k\mid s}
       \Prob\bigl(\mathcal G_D\cap
       \{\mathcal C_k(\A_D,z,H_0)\}\mid S=s\bigr)\\
 &\quad\leq C_{k,c}\tau(s)^k(1+H_0)^{k-1}D^{-\eta}\\
 &\quad\leq C_{k,c}\exp(k\sqrt{\log T})
                    (2+2\log D)^{k-1}D^{-\eta}.
 \end{split}
\]
The last expression is independent of $s$.
The logarithm of the factors preceding $D^{-\eta}$, apart from
the fixed constant, is
$k\sqrt{\log T}+(k-1)\log(2+2\log D)=o(\log D)$.
For sufficiently large $y$, it is at most $\eta\log D/2$.
Thus the conditional probability is at most
$C_{k,c}D^{-\eta/2}$ for every such $s$.

Average this bound over the realizations of $S$ for which
$\mathcal S_y$ holds, and then add the probabilities of the three
complementary events. This gives
\[
 \Prob(\mathcal E_y(a))
 \leq\Prob(\mathcal S_y^c)+\Prob(\mathcal B_y^c)
      +\Prob(\mathcal G_D^c)+C_{k,c}D^{-\eta/2}.
\]
All four terms tend to zero. This proves
$\Prob(\mathcal E_y(a))\to0$ under \eqref{eq:exponent-choice}.

If $\alpha_k>\beta_k/(1-\beta_k)$, the definition of $\alpha_k$
supplies an exponent $b>\beta_k/(1-\beta_k)$ at which the density-one
property holds. Choose $a$ with $\beta_k/(1-\beta_k)<a<b$.
The first part of the proof gives $\Prob(\mathcal E_y(a))\to1$.
On the other hand,
$a>\beta_k/(1-\beta_k)$ is equivalent to
$a/(1+a)>\beta_k$, so the second part gives
$\Prob(\mathcal E_y(a))\to0$.
This contradiction proves
$\alpha_k\leq\beta_k/(1-\beta_k)$.
Together with \eqref{eq:known-lower}, it completes the proof.
\end{proof}

\section{Concluding Remark}
We recently became aware of the preprint of Mao and Song \cite[Theorems~1.2 and~1.5]{MS}, who also establish the threshold identity in Theorem~\ref{thm:main} and the equality of the weak and strict entropy thresholds. They further show that, for every fixed $k\geq2$ and every $a>\beta_k/(1-\beta_k)$, the integers possessing $k$ distinct divisors in a multiplicative interval of relative length $(\log n)^{-a}$ have natural density zero. Both papers use flag refinement and entropy concavity, and \cite[Theorem~3.6]{MS} gives an affine subset-sum estimate related to Lemma~\ref{lem:translated}. Our argument gives a shorter probabilistic proof of the threshold identity: its arithmetic transfer passes through a random integer with independent geometric prime exponents and a direct coupling with the logarithmic random set.

Theorem~\ref{thm:main} gives the asymptotic equivalence $\alpha_k\sim\beta_k$ as $k\to\infty$.
Indeed, the bound $\alpha_k\leq(\log2)/(k+1)$ of Maier and Tenenbaum \cite{MT09} gives $\alpha_k\to0$, and the theorem gives $\beta_k=\alpha_k/(1+\alpha_k)$.
Consequently, the exponent introduced in \cite{FKTDelta} can be written as 
\[
 \eta_*:=\liminf_{k\to\infty}\frac{\log k}{\log(1/\beta_k)}
       =\liminf_{k\to\infty}\frac{\log k}{\log(1/\alpha_k)}
\]
By Theorem~\ref{thm:main}, this exponent can be defined equally in terms of divisor spacing or equal subset sums.

This exponent occurs in the study of the Erd\H{o}s--Hooley function $\Delta(n)=\sup_{u\in\R}\#\{d\mid n:u<\log d\leq u+1\}$.
Ford, Green and Koukoulopoulos \cite[Theorem~3]{FGK} give the lower bound $\Delta(n)\geq(\log\log n)^{\eta_*-\eps}$ for almost every $n$, for each fixed $\eps>0$.
Ford, Koukoulopoulos and Tao \cite[Theorem~2]{FKTDelta} obtain the corresponding mean-value bound $\sum_{n\leq x}\Delta(n)\gg_\eps x(\log\log x)^{1+\eta_*-\eps}$; see \cite{BTDelta,KTDelta} for upper bounds on the mean value. By Theorem~\ref{thm:strict}, $\eta_*$ can also be expressed using either entropy threshold. Determining the optimal flags and measures would give exact values of $\alpha_k$ for fixed $k$ and determine the exponent $\eta_*$ appearing in these concentration bounds.

\section*{Declaration on the use of AI}
The authors used generative AI tools to assist in discussing proof strategies, checking proofs, and improving exposition. All mathematical arguments, results, and conclusions were reviewed and verified by the authors.

\end{document}